\documentclass[12pt]{amsart}

\usepackage[margin=1in]{geometry}  % set the margins to 1in on all sides
\usepackage{latexsym}
\usepackage{amssymb}
\usepackage{amsmath}
\usepackage{amsthm}
\usepackage{amsfonts}
\usepackage{tikz}
\usepackage{url}
\usepackage{array}
\usepackage[all]{xy}
\usepackage{hyperref}

\newtheorem{thm}{Theorem}[section]
\newtheorem{lem}[thm]{Lemma}

\newtheorem{cor}[thm]{Corollary}

\newtheorem{definition}[thm]{Definition}
\newtheorem{exmp}[thm]{Example}

\newcommand{\scat}[1]{\mathbb{#1}}

\newcommand{\EC}[1]{\chi\left({#1}\right)}

\newcommand{\mr}[1]{\mathrm{#1}}
\newcommand{\ob}{\mr{Ob}\,}        
\newcommand*\mycirc[1]{%sabarthablog.blogspot.com
\tiny
  \begin{tikzpicture}
      \node[draw,circle,inner sep=1pt] {#1};
   \end{tikzpicture} \normalsize}
\newcommand{\Hom}{\mathrm{Hom}}
\begin{document}

\newcolumntype{x}[1]{>{\centering\arraybackslash\hspace{0pt}}p{#1}}

%\keywords
%\subjclass

\title[Euler characteristic]{Euler characteristic for weak n-categories and $\left(\infty,1\right)$-categories}
\author[A. Gonzalez, G. Necoechea, A. Stratmann]{Alex Gonzalez
\and Gabe Necoechea
\and Andrew Stratmann}
\thanks{The authors were partially supported by the Kansas State University Summer Undergraduate Mathematics Research program and NSF grant DMS-1262877}
%\address
%\date{\today}

\maketitle

%\tableofcontents

\begin{abstract}
	The Euler characteristic was defined for finite strict n-categories by Leinster using the theory of enriched categories. This was an extension of some of his earlier work, which defined Euler characteristic for finite categories. Building on Leinster's work, we extend the notion of Euler characteristic to certain finite weak 2-categories and present a sketch of a similar extension to weak n-categories. We also discuss the extension of the Euler characteristic to certain finite $\left(\infty,1\right)$-categories. 
\end{abstract}

\section*{Introduction}
Since first being used to classify polyhedra over three hundred years ago, the Euler characteristic has found applications in and extensions to various combinatoric and topological settings. This has made the Euler characteristic a useful topological invariant in the field of algebraic topology, where algebraic and combinatoric models of topological spaces are often sought. In recent times, categories and even higher categories have been used as (building blocks for) models of topological spaces, motivating a notion of Euler characteristic for categories and higher categories. Additionally, the Euler characteristic has properties which are analogous to those of cardinality, and it can be considered a fundamental notion of size. Viewed as a generalization of the cardinality of (finite) sets, it is natural to attempt to define the Euler characteristic for finite categories, which may be thought of as one level above finite sets in a hierarchy of ``categorification.'' In 2008, Leinster extended the Euler characteristic to a class of finite categories in \cite{LEICAT}. In \cite{LEIENRICH}, he introduced the concept of magnitude for enriched categories, which included the notion of Euler characteristic for certain strict n-categories $\left(n < \infty\right)$. This general notion of magnitude is invariant under equivalence of enriched categories and also well-behaved with respect to both products and coproducts. 
While Leinster's aforementioned work addresses a substantial portion of the (finite) categorical landscape, there remain mathematical objects outside of the scope of his work which are of significant interest. In particular, within higher category theory, there is significant interest in weak n-categories $\left (n < \infty\right)$ and $\left(\infty,1\right)$-categories. To the best of our knowledge, neither of these can be handled as the enriched categories for which \cite{LEIENRICH} provides a notion of magnitude. In what follows, we will develop the Euler characteristic in higher category theory. In particular, we present a setup for the Euler characteristic of a weak n-category and show that this is invariant under an appropriate notion of equivalence for n-categories. We also introduce the reader to $\left(\infty,1\right)$-categories and some of the beginning steps toward Euler characteristic for these creatures. 

\section{Basic Definitions}
%\label{sec:Basic Definitions}
\subsection{Category Theory}
	A \emph{(small) category} $\scat{A}$ consists of the following data:
	\begin{enumerate}
	\item A set $\ob\left(\scat{A}\right)$ of objects;
	\item For all $\left(x,y\right) \in \ob\left( \scat{A}\right) \times \ob\left( \scat{A}\right)$, a set $\Hom\left(x,y\right)$ of morphisms (or arrows)  from x to y, sometimes called the Hom-set from x to y;
	\item A composition rule $\circ \colon \Hom\left(y,z\right) \times \Hom\left(x,y\right) \to \Hom\left(x,z\right)$ for all $x, y, z \in \ob \left(\scat{A}\right)$;
	\end{enumerate} 
	and satisfies the following axioms:
	\begin{enumerate}
	%\item For $f \in \Hom\left(x,y\right), g \in \Hom\left(y,z\right), h \in \Hom\left(z,w\right), h \circ \left(g \circ f\right) = \left(h \circ g\right) \circ f$. 
	\item Whenever such compositions make sense, $h \circ \left(g \circ f\right) = \left(h \circ g\right) \circ f$;
	\item For every $x \in \ob\left( \scat{A}\right)$, there is a morphism $1_x \in \Hom\left(x,x\right)$ such that $f \circ 1_x = f$ for all $f \in \Hom\left(x,y\right)$ and $1_x \circ g = g$ for all $g \in \Hom\left(z,x\right)$. We call $1_x$ the \emph{identity} morphism (on $x$). 
	\end{enumerate}
	Henceforth, small categories will generally be referred to simply as ``categories.'' It should be noted that $\Hom\left(x,y\right)$ may be empty, in which case there are no morphisms from $x$ to $y$, though there may be morphisms from $y$ to $x$. Any morphism in $\Hom\left(x,x\right)$ is called an \emph{endomorphism} of $x$. Additionally, we denote by $\Hom\left(\scat{A}\right)$ the collection of all morphisms in the category. If $|\Hom\left(\scat{A}\right)|$ is finite then we say $\scat{A}$ is \emph{finite}. The existence of an identity morphism for each object allows us to bound the number of objects from above by $|\Hom\left(\scat{A}\right)|$, thus a finite category has both a finite number of arrows and a finite number of objects. (This is, in fact, the ``right'' definition among the two natural choices: requiring only finitely many objects does not prevent us from having infinitely many morphisms.) \par
	We now present a few examples of categories. When we present diagrams of categories, the identity morphisms will be omitted. 
	
	\begin{exmp}
	\label{exmp: 1 category}
	The category $\mathbf{1}$ consists of one object, $*$, and its identity morphism. 
	\end{exmp}
	
	\begin{exmp}
	\label{exmp: terminal or initial}
	$\xymatrix{x \ar[r]^{f} & y }$ \\ is a category with two objects, $x$ and $y$. It has three morphisms: $1_{x} \in \Hom\left(x,x\right)$, $1_{y} \in \Hom\left(y,y\right)$, and $f \in \Hom\left(x,y\right)$. There are no morphisms in $\Hom\left(y,x\right)$. 
	\end{exmp}
	
	\begin{exmp}
	\label{exmp: three object}
	$\xymatrix{x \ar@/^/[r]|{f} \ar@/_2pc/[rr]|{h \circ f} & y \ar@/^/[l]|{g} \ar[r]^{h} & z}$
	\end{exmp} 
	
	\begin{exmp}
	\label{exmp: set as category}
	A set $S$ can be considered as a category by taking the elements of $S$ as the objects and the identity morphisms as the only morphisms. In general, a category whose only morphisms are identities is called a \emph{discrete category}.
	\end{exmp}
	
	\begin{exmp}
	\label{exmp: product}
	Given two categories, $\scat{A}$ and $\scat{B}$, we can construct a product category $\scat{A} \times \scat{B}$ which, for $\scat{A}$ and $\scat{B}$ small, has the following data:
	\begin{enumerate}
	\item $\ob\left(\scat{A} \times \scat{B}\right) = \left\{\left(a,b\right)\colon a \in \ob\left( \scat{A}\right), b \in \ob \left(\scat{B}\right)\right\}$
	\item For all $\left(x,y\right)$ and $\left(z,w\right) \in \ob \left(\scat{A} \times \scat{B}\right)$, \[\Hom\left(\left(x,y\right),\left(z,w\right)\right) = \\ \left\{\left(f,g\right)\colon f \in \Hom\left(x,z\right), g \in \Hom\left(y,w)\right)\right\}. \]
	\end{enumerate}
 Composition of morphisms is ``component-wise,'' given by composition in $\scat{A}$ in the first component and $\scat{B}$ in the second component. 
	\end{exmp}
	
	\begin{exmp}
	\label{exmp: coproduct}
	Given two categories, $\scat{A}$ and $\scat{B}$, we can construct a disjoint union category $\scat{A} \coprod \scat{B}$, which for $\scat{A}$ and $\scat{B}$ small, has the following data:
	\begin{enumerate}
	\item $\ob\left( \scat{A}\right) \coprod \scat{B} = \ob\left( \scat{A}\right) \coprod \ob\left( \scat{A}\right)$ is just the disjoint union of the object sets.
	\item $\Hom_{\scat{A} \coprod \scat{B}}\left(x,y\right) = 
		\begin{cases} 
		\Hom_{\scat{A}}\left(x,y\right) & x,y \in \ob\left( \scat{A}\right) \\
		\Hom_{\scat{B}}\left(x,y\right) & x,y \in \ob\left(\scat{B}\right) \\ 
		\varnothing & otherwise
		\end{cases}$
	\end{enumerate}
	\end{exmp}
	
	\begin{exmp}
	\label{exmp: monoid as category}
	We can treat a monoid as a category with one object, $*$. The morphisms in $\Hom\left(*,*\right)$ are the elements of the monoid, and the composition of morphisms follows the composition rule of the monoid. The identity element is the identity morphism $1_{*}$. 
	\end{exmp}
	
	\begin{exmp}
	\label{exmp: opposite category}
	Given a category $\scat{A}$, we have another category $\scat{A}^{op}$, the \emph{opposite category} of $\scat{A}$. This category has the same objects and composition rule as $\scat{A}$, but the direction of each morphism is reversed.
	\end{exmp}
	
	\begin{definition}
	An object $x$ in a category $\scat{A}$ is called \emph{terminal} $\left(\text{resp.} \ \emph{initial}\right)$ if $|\Hom\left(w,x\right)| = 1 \left(\text{resp.} \ |\Hom\left(x,w\right)| = 1 \right)$ for all $w \in \ob\left( \scat{A}\right)$. 
	\end{definition}

	\begin{definition}
	Two objects, $x$ and $y$ are \emph{isomorphic} in $\scat{A}$ if there exist morphisms $f \in \Hom\left(x,y\right)$ and $g \in \Hom\left(y,x\right)$ such that $f \circ g = 1_y$ and $g \circ f = 1_x$. When such $f$ and $g$ exist, we call them \emph{isomorphisms}. 
	\end{definition}
	
	\begin{exmp}
	\label{exmp: group as category}
	A group is just a monoid with inverses for every element, so we can view a group as a one-object category in which all the arrows are isomorphisms. 
	\end{exmp}
	
	\begin{definition}
	A \emph{skeleton} or \emph{skeletal subcategory} $\scat{A}_o$ of $\scat{A}$ is a full subcategory whose object set is a set of representatives of the isomorphism classes of $\scat{A}$. 
	\end{definition}
	% hspace works better than kern in the next example. 
	\begin{exmp}
	\label{exmp: skeleta}
	Let $\scat{A}$ be determined by the diagram \hspace{-3em} $\xymatrix{&x \ar@/^/[r]|{f} \ar@/_2pc/[rr]|{h \circ f} & y \ar@/^/[l]|{g} \ar[r]^{h} & z}$ along with the identity morphisms. Since the only endomorphisms of x and y are identities, $g \circ f$ = $1_{x}$ and $f \circ g = 1_{y}$, i.e. $x$ and $y$ are isomorphic objects. We have two possible skeletons of $\scat{A}\colon$ 
	\begin{enumerate}
	\item \hspace{-3em} $\xymatrix{&x \ar[r]^{h \circ f} &z}$
	\item \hspace{-3em} $\xymatrix{&y \ar[r]^{h} &z}$
	\end{enumerate}
	\end{exmp}
	
	From the above example, we might be motivated to say that the two skeleta in Example 1.4 are essentially the same. We can make this notion precise for any two categories by introducing the notions of functors and natural isomorphisms.
	
	\begin{definition}
	A \emph{functor} $F\colon\scat{A}\to\scat{B}$ is a ``map between categories'' which takes objects in $\scat{A}$ to objects in $\scat{B}$ ($x \mapsto F\left(x\right)$)  and morphisms in $\scat{A}$ to morphisms in $\scat{B}$ ($f \mapsto F\left(f\right)$) and respects composition in the sense that
	\begin{enumerate}
	\item $F\left(1_x\right) = 1_{Fx}$
	\item $F\left(f\right) \circ F\left(g\right) = F\left(f \circ g\right)$ whenever one (and hence both) of these compositions exist. 
	\end{enumerate}
	\end{definition}

	To be slightly more thorough, we can consider a functor $F\colon\scat{A}\to\scat{B}$ as a collection of the following functions: 
	\begin{itemize}
	\item $F_{0}\colon\ob\left(\scat{A}\right)\to\ob\left(\scat{B}\right)$
	\item For each $\left(x,y\right) \in \ob\left( \scat{A}\right) \times \ob\left( \scat{A}\right)$, a function $F_{xy}\colon \Hom\left(x,y\right)\to \Hom\left(F_{0}\left(x\right),F_{0}\left(y\right)\right)$. 
	\end{itemize}
	In general, which function in the collection we are using will be understood, so we will omit subscripts like when working with functors. We will often specify a functor $F\colon \scat{A} \to \scat{B}$ by writing what it does to a arbitrary object $x$ of $\scat{A}$ and what it does to an arbitrary morphism of $\scat{A}$. For instance, we will write that $F \colon \scat{A} \to \scat{B}$ is given by 
	\begin{itemize}
	\item $x \mapsto F\left(x\right)$
	\item $f \mapsto F\left(f\right)$
	\end{itemize}
	
	\begin{exmp}
	\label{exmp: identity functor}
	Any category $\scat{A}$ has an identity functor $Id_{\scat{A}}\colon\scat{A}\to\scat{A}$ given by 
	\begin{itemize}
	\item $x \mapsto x$
	\item $f \mapsto f$
	\end{itemize}
	\end{exmp}
	
	\begin{exmp}
	\label{exmp: object identification functor}
	For an object $x$ in a category $\scat{A}$, we have a functor $\bar{x}\colon\mathbf{1}\to\scat{A}$ which ``picks out'' the object $x$. We will use this example in our definition of bicategories. 
	\end{exmp}
	
	\begin{exmp}
	\label{exmp: functor composition}
	Given functors $F\colon\scat{A}\to\scat{B}$ and $G\colon\scat{B}\to\scat{C}$, we have a functor $G \circ F\colon\scat{A}\to\scat{C}$ given by 
	\begin{itemize}
	\item $x \mapsto G\left(F\left(x\right)\right)$
	\item $f \mapsto G\left(F\left(f\right)\right)$
	\end{itemize}
	\end{exmp}
	
	\begin{exmp}
	\label{exmp: product as functor}
	The product category $\scat{A} \times \scat{B}$ can be thought of as an element of the image of (the object function of) a functor $\times\colon\ \mathbf{Cat}\times\mathbf{Cat}\to\mathbf{Cat}$, where $\mathbf{Cat}$ is the category whose objects are (small) categories and whose morphisms are functors between them. We have associativity (up to isomorphism) when taking products, and we will occasionally refer to the product on a finite family of categories $\left\{\scat{A}_i\right\}$ by $\prod_{i} \scat{A}_i$.
	\end{exmp}
	
	\begin{exmp}
	\label {exmp: coproduct as functor}
	The coproduct category $\scat{A} \coprod \scat{B}$ can be thought of as the object image of a functor $\coprod\colon \mathbf{Cat}\times\mathbf{Cat}\to\mathbf{Cat}$. As in the case of products, we have associativity (up to isomorphism) when taking coproducts, and we will use $\coprod_{i} \scat{A}_i$ to denote the coproduct on a finite family of categories $\left\{\scat{A}_i\right\}$. 
	\end{exmp}
	% NOT FINISHED!
	In Examples \ref{exmp: product as functor} and \ref{exmp: coproduct as functor}, we are really seeing a special case of more general constructions of product and coproduct, which return a unique (up to isomorphism) object for two objects of a category (when such an object exists). These constructions are, in turn, special cases of limits and colimits. For a more rigorous and thorough treatment of these concepts, we refer the reader to \cite{MACCAT}.

	\begin{definition}
	A \emph{natural isomorphism} $\eta$ is a collection of isomorphisms in $\Hom\left(\scat{B}\right)$ that essentially tells us that two functors are the same ($F \cong G$). More formally, two functors $F\colon\scat{A}\to\scat{B}$ and $G\colon\scat{A}\to\scat{B}$ are naturally isomorphic if, for any two objects $x, y$ in $\scat{A}$, we have morphisms $\eta_{x}\colon F\left(x\right) \to G\left(x\right)$ and $\eta_{y}\colon F\left(y\right) \to G\left(y\right)$ such that, for any morphism $f\colon x \to y$, the following diagram commutes:\\
	\begin{center}
	$\xymatrix@=5em{& F\left(x\right)\ar[d]_{F\left(f\right)} \ar[r]^{\eta_{x}} & G\left(x\right) \ar[d]^{G\left(f\right)} & \\ & F\left(y\right) \ar[r]^{\eta_{y}} & G\left(y\right)}$
	\end{center}
	i.e. $\eta_{y} \circ F\left(f\right) = G\left(f\right) \circ \eta_{x}$ for all $f\colon x \to y$. 
	\end{definition}
	
	If we require the weaker condition that $\eta$ is a collection of morphisms (not necessarily isomorphisms) and keep everything else the same, we say $\eta$ is a \emph{natural transformation}. This gives us another example of a category, the functor category.

	\begin{exmp}
	\label{exmp: functor category}
	Let $\scat{A}$ be a small category. Then $[\scat{A},\scat{B}]$ is a category whose objects are functors from $\scat{A}$ to $\scat{B}$ and whose morphisms are natural transformations between functors. 
	\end{exmp}

	We can now formalize our notion of equivalence for categories. 
	\begin{definition}
	Two categories $\scat{A}$ and $\scat{B}$ are \emph{equivalent} if there exist functors $F\colon\scat{A}\to\scat{B}$ and $G\colon\scat{B}\to\scat{A}$ such that $G \circ F \cong Id_{\scat{A}}$ and $F \circ G \cong Id_{\scat{B}}$. 
	\end{definition}
	Alternatively, if both categories $\scat{A}$ and $\scat{B}$ are small, we can say that $\scat{A}$ and $\scat{B}$ are equivalent if $G \circ F$ and $Id_{\scat{A}}$ are isomorphic objects in $[\scat{A},\scat{A}]$ and $F \circ G$ and $Id_{\scat{B}}$ are isomorphic objects in $[\scat{B},\scat{B}]$. When $\scat{A}$ and $\scat{B}$ are equivalent through $F\colon\scat{A}\to\scat{B}$ and $G\colon\scat{B}\to\scat{A}$, we shall also write that \kern-3em $\xymatrix{&\scat{A} \ar@<.5ex>[r]^{F} &\scat{B} \ar@<.5ex>[l]^{G}}$ is an equivalence. 
	
	\subsection{Bicategory Theory}
	In what follows, we describe small bicategories. See \cite{LEIBICAT} for a more detailed, but still brief, exposition. Much of what follows is nearly identical to Leinster's survey. We continue working in the ``small'' setting. 
	\begin{definition}
	A \emph{bicategory} $\scat{A}$ consists of the following data: 
	\begin{enumerate}
	\item A set of objects $\ob\left( \scat{A}\right)$. Elements in $\ob\left( \scat{A}\right)$ are sometimes called \emph{0-cells}. 
	\item For any $\left(x,y\right) \in \ob\left( \scat{A}\right) \times \ob\left( \scat{A}\right)$, a small category $\scat{A}\left(x,y\right)$ which consists of
	\begin{itemize}
	\item Objects called \emph{1-cells}
	\item Arrows between 1-cells which are called \emph{2-cells}. The identity 2-cell for a 1-cell $f$ will be denoted $\widetilde{1_{f}}$. 
	\item A composition rule $\mycirc{v}_{xy}$. 
	\end{itemize}
	\item For any triple $\left(x,y,z\right)$ of objects in $\ob\left( \scat{A}\right)$, a functor $\mycirc{h}_{xyz}\colon \scat{A}\left(y,z\right) \times \scat{A}\left(x,y\right) \to \scat{A}\left(x,z\right)$ given by
	\begin{itemize}
	\item $\left(g, f\right) \mapsto g \circ f = gf$
	\item $\left(\beta, \alpha\right) \mapsto \beta \ \mycirc{h} \ \alpha$
	\end{itemize}
	\item For any object $x \in \ob\left( \scat{A}\right)$, a functor $I_{x} \colon \textbf{1} \to \scat{A}\left(x,x\right)$ which identifies a special, ``almost identity'' 1-cell, $1_x$.  
	\item Additionally, we have some natural isomorphisms on certain products of functors which tell us how our functors $\mycirc{h}_{xyz}$ work with ``almost identities'' and associativity. For brevity, we omit the diagrams (which may be found in \cite{LEIBICAT}), and only state the important consequences of them:
	\begin{enumerate}
	\item For 1-cells $f, g, h$ such that $\left(hg\right)f$ and hence $h\left(gf\right)$ are well-defined, we have a 2-cell isomorphism $A_{hgf} \colon \left(hg\right)f \to h\left(gf\right)$. This 2-cell isomorphism (or \emph{2-isomorphism}) is called the \emph{associator of hgf}. 
	\item For a 1-cell $f \in \ob\left( \scat{A}\right)\left(x,y\right)$, we have a 2-isomorphism $r_{f} \colon f \circ 1_x \to f$. This 2-isomorphism is called the \emph{right unitor of f}.
	\item For a 1-cell $g \in \ob\left( \scat{A}\right)\left(y,x\right)$, we have a 2-isomorphism $\ell_{f} \colon 1_x \circ f \to f$. This 2-isomorphism is called the \emph{left unitor of f}. 
	\end{enumerate}
	\end{enumerate}
	\end{definition}
	When it is clear from context, we just refer to these 2-isomorphisms as associators and right or left unitors without specifying which 1-cells they are related to. If all of these 2-isomorphisms are exact identity 2-cells, then our bicategory is called \emph{strict}. Otherwise, our bicategory is called \emph{weak}. There is a good way of coming up with strict 2-categories. Namely, for all $x, y \in \scat{A}$ we can replace $\Hom\left(x,y\right)$ with an object $\scat{A}\left(x,y\right)$ from a special sort of category $M$ known as a monoidal category. (We say we have enriched our category over M, and call $\scat{A}\left(x,y\right)$ the Hom-object from x to y.) The category of all small categories is a monoidal category, so we can enrich a category over $\mathbf{Cat}$ to get a strict 2-category. When we generalize the notion of bicategory to weak n-category, we will again have strict counterparts that we can produce through enrichment: a strict n-category will be a category enriched over the category of strict (n-1)-categories.\footnote{The category $\mathbf{Set}$ is also a monoidal category, and a category enriched over $\mathbf{Set}$ is just a usual category. In fact, we may begin our inductive process of strict n-categories by considering $\mathbf{Set}$ to be our strict 0-category.}
	
	\begin{exmp}
	\label{defn: Cat as bicat}
	We can think of $\mathbf{Cat}$ as a bicategory which has all small categories as its objects, functors between the categories as its 1-cells, and natural transformations between functors as its 2-cells. This is often said to be the prototypical example of a bicategory. (In fact, it is a strict 2-category.)
	\end{exmp}
	
	\begin{definition}
	Two 0-cells $x$ and $y$ of a bicategory are \emph{internally equivalent} if there exists $f \in \scat{A}\left(x,y\right)$ and $g \in \scat{A}\left(y,x\right)$ such that 
	\begin{enumerate}
	\item There exists a 2-isomorphism $\alpha \colon g \circ f \to 1_x$
	\item There exists a 2-isomorphism $\beta \colon f \circ g \to 1_y$.
	\end{enumerate}
	\end{definition}
	This definition may be thought of as a weakening of the conditions of isomorphic objects in a category. In some sense, the motivation is that we now have ``higher levels'' of equivalence that we can check. Indeed, we may think of the notion of isomorphic objects as a weakening of the equality of elements in a set. %For example, $\alpha$ is an equivalence in the category $\scat{A}\left(x,x\right)$, which means that it is an isomorphism, which itself implies the notion of equality

	In addition to a notion of equivalence of objects in a bicategory, we have a notion of equivalence of bicategories themselves. Recall that an equivalence of categories involved two functors and two natural transformations between the two compositions of these functors and the identity functors for the categories. Equivalence for bicategories will work out in roughly the same way. The analogue of a functor in bicategory theory is, somewhat unfortunately, called a morphism. 
	
	\begin{definition}
	\label{defn: morphism between bicats}
	A \emph{morphism} $F\colon\scat{A}\to\scat{B}$ between bicategories $\scat{A}$ and $\scat{B}$ consists of the following data:
	\begin{enumerate}
	\item An object function $F_{0}\colon \ob\left( \scat{A}\right) \to \ob\left( \scat{B}\right)$
	\item For each pair $\left(x,y\right) \in \ob\left( \scat{A}\right) \times \ob\left( \scat{A}\right)$, a functor $F_{xy}\colon\scat{A}\left(x,y\right) \to \scat{B}\left(F_{0}\left(x\right),F_{0}\left(y\right)\right)$.
	\end{enumerate}
	and respects ``composition'' in the following sense
	\begin{enumerate}
	\item There is a 2-cell $\phi_{gf}$ in $\scat{A}\left(F_{0}\left(x\right),F_{0}\left(z\right)\right)$ between $F_{yz}\left(g\right) \circ F_{xy}\left(f\right)$ and $F_{xz}\left(g \circ f \right)$. 
	\item There is a 2-cell $\phi_{x}$ in $\scat{A}\left(F_{0}\left(x\right),F_{0}\left(x\right)\right)$ between  $I^{\prime}_{F_{0}\left(x\right)}$, the almost-identity 1-cell in $F_{0}\left(x\right)$, and $F_{0}\left(I_{x}\right)$.
	\end{enumerate}
	
	\end{definition}
	In addition to the data above, morphisms also satisfy some additional diagrams that can be found in \cite{LEIBICAT}. As in the case of functors, we will generally omit the subscripts from the functions and functors that make up a morphism of bicategories because it will be clear from context; additionally, every bicategory $\scat{A}$ has an identity morphism $Id_{\scat{A}}$ which has identity object function and every functor being the appropriate identity functor. Notice that the 2-cells which relate composition under a morphism of bicategories are not very ``strong.'' They are not strict equalities, and they need not even be isomorphisms. If all of them \emph{are} isomorphisms, we follow Leinster and call the morphism a \emph{homomorphism}.  We will use homomorphisms in our definition of biequivalence. We also need to talk about transformations and modifications.
	
	\begin{definition}
	\label{defn: transformations of morphisms}
	A \emph{transformation} $\sigma$ between morphisms $F\colon\scat{A}\to\scat{B}$ and $G\colon\scat{A}\to\scat{B}$ is the analogue to a natural transformation between functors. Specifically, it gives
	\begin{enumerate}
	\item A 1-cell $\sigma_{x}$ between $Fx$ and $Gx$ for each $x \in \ob\left( \scat{A}\right)$. 
	\item A 2-cell $\sigma_{f}$ between $Gf \circ \sigma_{x}$ and $\sigma_{y} \circ Ff$ for all 1-cells $f\colon x \to y$. 
	\end{enumerate}
	\end{definition}
	
	The analogy between a transformation $\sigma$ of morphisms $F^{\prime}$ and $G^{\prime}$ and a natural transformation $\eta$ of functors $F$ and $G$ is worth remarking on. Notice that the assignment of the 1-cell $\sigma_{x^{\prime}}$ to a 0-cell $x^{\prime}$ of a bicategory is essentially identical to the process by which $\eta_{x}$, a morphism, is assigned to an object $x$ of a category. Recall also that, in a natural transformation, we have commutative diagrams which expresses that $Gf \circ \eta_{x} = \eta_{y} \circ Ff$ for each $f\colon x \to y$. In our transformation of morphisms, this notion is weakened: all we say is that there is a 2-cell $\sigma_{f^{\prime}}$ between $G^{\prime}f^{\prime} \circ \sigma_{x^{\prime}}$ and $\sigma_{y^{\prime}} \circ F^{\prime}f^{\prime}$ for each $f^{\prime} \colon x^{\prime} \to y^{\prime}$. If $\sigma_{f^{\prime}}$ is the identity for each $f^{\prime}$ then we have essentially recovered the idea of natural transformation, but we still need to keep in mind that this is a transformation between morphisms, not functors. At any rate, we will call such a transformation a \emph{strict transformation}. In between transformations and strict transformations, there are \emph{strong transformations}, which are transformations for which each $\sigma_{f^{\prime}}$ is an isomorphism. It is the strong transformations that we will use to discuss equivalence of bicategories. We need an additional piece of the puzzle called a modification.
	
	\begin{definition}
	\label{defn: modifications between transformations}
	A modification $\Gamma$ between transformations $\sigma$ and $\sigma^{\prime}$ (from a morphism $F\colon \scat{A} \to \scat{B}$ to a morphism $G \colon \scat{A} \to \scat{B}$) is a collection of 2-cells $\Gamma_{x}\colon \sigma_{x} \Rightarrow \sigma^{\prime}_{x}$ for each $x \in \ob\left( \scat{A}\right)$. Additionally, for any $f\colon x \to y$, the following diagram commutes:
	\begin{center}
	$\xymatrix@=5em{& Gf \circ \sigma_{x}\ar[d]_{\sigma_{f}} \ar[r]^{1 \mycirc{h} \Gamma_{x}} & Gf \circ \sigma^{\prime}_{x} \ar[d]^{\sigma^{\prime}f} & \\ & \sigma_{y} \circ Ff \ar[r]^{\Gamma_{y} \mycirc{h} 1} & \sigma^{\prime}y \circ Ff}$
	\end{center}
	\end{definition}
	
	This should look similar to natural transformations between functors, almost as though it were a natural transformation of natural transformations; indeed, why not have a similar way of tracking what the natural transformations (now generalized to transformations in the higher category setting) look like relative to one another? Modifications allow us to do this: we move up one extra level of comparison. Just as we had a notion of equivalence of categories come out of natural isomorphisms, we would like a ``next-level'' equivalence of bicategories to come out of \emph{strong modifications}, which are modifications for which each $\Gamma_{x}$ is an isomorphism. \par
	Recall that our alternative definition of equivalence of (small) categories $\scat{A}$ and $\scat{B}$ was just a statement about the existence of certain isomorphisms in the categories $[\scat{A},\scat{A}]$ and $[\scat{B},\scat{B}]$ (as defined in Example \ref{exmp: functor category}). We would like to do something similar for equivalence of bicategories, but now we will want a statement about the existence of certain \emph{internal equivalences}. But what bicategory should these internal equivalences be in? The answer is the following bicategory, which is another ``prototypical'' example of a bicategory. 
	
	\begin{exmp}
	\label{exmp: homomorphism bicat} %homomorphisms seem to be called pseudofunctors quite a lot. . .
	For (small) bicategories $\scat{A}$ and $\scat{B}$, we have a bicategory $[\scat{A},\scat{B}]$ whose objects are homomorphisms, 1-cells are strong transformations, and 2-cells are strong modifications. % I think smallness of bicategories gives you enough. . .
	\end{exmp}
	
	It is now relatively clear what an equivalence of bicategories should look like, at least for small bicategories, using Example \ref{exmp: homomorphism bicat}.
	
	\begin{definition}
	Two (small) bicategories $\scat{A}$ and $\scat{B}$ are equivalent if there exist homomorphisms $F\colon\scat{A}\to\scat{B}$ and $G\colon\scat{B}\to\scat{A}$ such that $G \circ F$ is internally equivalent to $Id_{\scat{A}}$ in $[\scat{A},\scat{A}]$ and $F \circ G$ is internally equivalent to $Id_{\scat{B}}$ in $[\scat{B},\scat{B}]$.
	\end{definition}
	
	It is not currently clear what the best way to generalize to the notion of weak n-category is. Several competing definitions have been put forward, and it is not yet known if they are all equivalent. (For a presentation of ten such definitions, see \cite{LEISURVEY}.) For our purposes, we will consider an n-category to be something with 0-cells, 1-cells, 2-cells, . . ., n-cells and which has a notion of internal equivalence involving n-isomorphisms. 
	
	\section{The Euler characteristic for n-categories}
	\subsection{The Construction of the Euler Characteristic for 1-categories}
	We now define the Euler characteristic of a finite 1-category as introduced by Leinster, present several example calculations, and review some results found of Leinster's. 
	\begin{definition}
	The \emph{adjacency matrix} of a finite category $\scat{A}$ under some total ordering $\left\{a_1, a_2, . . ., a_n\right\}$ is \[\emph{M}\left(\scat{A}\right) = \begin{bmatrix}
    |\Hom\left(a_1,a_1\right)|       & \dots & |\Hom\left(a_1,a_n\right)| \\
    \vdots & \ddots & \vdots \\
    |\Hom\left(a_n,a_1\right)|      & \dots & |\Hom\left(a_n,a_n\right)|
\end{bmatrix}. \]
	\end{definition}
	
	\begin{definition}
	\label{defn: weighting and coweighting}
	A \emph{weighting} on $\scat{A}$ is a solution $\mathbf{v} = \begin{bmatrix} k^{a_{1}} & k^{a_{2}} & \dots{} & k^{a_{n}} \end{bmatrix}^{\intercal}$ to the matrix equation $\emph{M}\left(\scat{A}\right)\mathbf{v} =  \begin{bmatrix} 1 & 1 & \dots & 1 \end{bmatrix}^{\intercal}$. A \emph{coweighting} on $\scat{A}$ is a solution $\mathbf{u} = \begin{bmatrix} k_{a_{1}} & k_{a_{2}} & \dots & k_{a_{n}} \end{bmatrix}$ to the matrix equation $\mathbf{u}\emph{M}\left(\scat{A}\right) =  \begin{bmatrix} 1 & 1 & \dots & 1 \end{bmatrix}$. Alternatively, a \emph{weighting} on $\scat{A}$ is a function $k^{\bullet}\colon \ob\left( \scat{A}\right) \to \mathbb{Q}$ such that $\sum_{b \in \ob\left( \scat{A}\right)} |\Hom\left(a,b\right)|k^{b} = 1$ for all $a \in \ob\left( \scat{A}\right)$ and a coweighting on $\scat{A}$ $k_{\bullet}$ is a function $k_{\bullet}\colon\ob\left( \scat{A}\right) \to \mathbb{Q}$ such that $\sum_{a \in \ob\left( \scat{A}\right)} |\Hom\left(a,b\right)|k_{a} = 1$ for all $b \in \ob\left( \scat{A}\right)$..
	\end{definition}
	
	Note that a coweighting on $\scat{A}$ can also be thought of as a weighting on $\scat{A}^{op}$, and a weighting on $\scat{A}$ can be thought of as a coweighting on $\scat{A}^{op}$. Of course, a category $\scat{A}$ need not have a weighting or coweighting, but if it has both then $\sum_{i=1}^{n} k^{a_{i}} = \sum_{i=1}^{n} k_{a_{i}}$, which gives us a definition for Euler characteristic of a finite category.
	
	\begin{definition}
	If a category $\scat{A}$ admits both a weighting $k^{\bullet}$ and coweighting $k_{\bullet}$ then its \emph{Euler characteristic} is $\EC{\scat{A}} =  \sum_{i=1}^{n} k^{a_{i}} = \sum_{i=1}^{n} k_{a_{i}}$. 
	\end{definition}
	
	Note that, as a convention, we shall define the Euler characteristic of the empty category to be zero.
	
	\begin{exmp}
	\label{calculation: terminal or initial}
	In Example \ref{exmp: terminal or initial}, our matrix is $\begin{bmatrix} 1 & 1 \\ 0 & 1 \end{bmatrix}$, which has weighting $\textbf{x} = \begin{bmatrix} 0 \\ 1 \end{bmatrix}$ and coweighting $\textbf{y} = \begin{bmatrix} 1 & 0 \end{bmatrix}$. Therefore, the Euler characteristic is $1$. 
	\end{exmp}
	
	\begin{exmp}
	\label{calculation: three object}
	In Example \ref{exmp: three object}, the adjacency matrix is $\begin{bmatrix} 1 & 1 & 1 \\ 1 & 1 & 1 \\ 0 & 0 & 1 \end{bmatrix}$. This has weighting $\begin{bmatrix} 0 \\ 0 \\ 1 \end{bmatrix}$ and coweighting $\begin{bmatrix} 1 & 0 & 0 \end{bmatrix}$. (In fact, it has infinitely many weightings and coweightings.) The Euler characteristic is 1. \par
	\end{exmp}
	
	In light of the last two examples, it is worth noting that how we order the objects in our category does not really matter for the purposes of these matrix equations--and thus also for the Euler characteristic. Additionally, in both examples we had initial objects and terminal objects. In Example \ref{calculation: terminal or initial}, the initial object was $x$, and the terminal object was $y$. In Example \ref{calculation: three object}, the initial objects were $x$ and $y$, and the terminal object was $z$. An initial object corresponds to a row of ones in the adjacency matrix, and a terminal object corresponds to a column of ones. In either case, this forces the sum of weightings and sum of coweightings to be 1. Thus the Euler characteristic of a category with an initial object or terminal object is $1$. 
	
	\begin{exmp}[\cite{LEICAT} Ex. 2.3a]
	\label{calculation: set as category}
	A set $S$ with $n$ elements can be viewed as a category $\scat{S}$ with $n$ objects and only identity morphisms. Its adjacency matrix is thus the $n \times n$ identity matrix, so its Euler characteristic is $ \EC{\scat{S}} = |\ob\left(\scat{S}\right)| = |S| = n$. 
	\end{exmp}
	
	\begin{exmp}[\cite{LEICAT} Ex. 2.3b]
	\label{calculation: monoid as category}
	Since a finite monoid $M$ is considered as a category $\scat{M}$ with only one object, the adjacency matrix of $\scat{M}$ is $1 \times 1$. Thus we have Euler characteristic given by the solution $x$ to $|\Hom\left(\scat{M}\right)| \cdot x = x \cdot |\Hom\left(\scat{M}\right)| = 1$. That is, the Euler characteristic of a monoid is $\EC{\scat{M}} = \frac{1}{|\Hom\left(\scat{M}\right)|}$.
	\end{exmp}
	
	\begin{exmp}
	\label{calculation: group as category}
	Note that, while we might check that our adjacency matrix describes an actual category (and not just a general directed multigraph with loops), the adjacency matrix does not really capture any of the behavior of the composition of morphisms. This general statement is illustrated here by considering that a finite group considered as a category $\scat{G}$ has $\EC{\scat{G}} = \frac{1}{|\Hom\left(\scat{G}\right)|}$, just as a finite monoid considered as a category $\scat{M}$ had $\EC{\scat{M}} = \frac{1}{|\Hom\left(\scat{M}\right)|}$.
	\end{exmp}
	%In what follows, I am not consistent in DMWL vs quiver. 
	We may view a finite category as a special kind of directed multigraph with loops. (Henceforth, a directed multigraph with loops will be called a \emph{quiver}.) Viewing a finite category as a special kind of quiver makes certain notions rather natural. For instance, we can talk about the connected components of a finite category $\scat{A}$. If each of these connected components has Euler characteristic, then the Euler characteristic of $\scat{A}$ is the sum of the Euler characteristics of the connected components. This is expressed in a more category-theoretic way by \cite[Proposition 2.6]{LEICAT}, which we state in the next two theorems.
	
	\begin{thm}
	Let $\left\{\scat{A}_i\right\}$ be a finite family of finite categories with Euler characteristic. Then $\coprod_{i} \scat{A}_i$ has Euler characteristic and $\EC{\coprod_{i} \scat{A}_i} = \sum_{i} \EC{\scat{A}_i}$. 
	\end{thm}
	
	In fact, the preceding theorem says more than what we have asserted in our comments leading up to it: not only can we calculate Euler characteristic of an existing category by splitting it up into calculations on connected components, but we can also group up disjoint finite categories with Euler characteristic into a new category and then calculate the new category's Euler characteristic from the Euler characteristics of its disjoint pieces. 
	
	Additionally, we have the Euler characteristic well-behaved under product as well, according to the following theorem of Leinster:
	
	\begin{thm}
	Let $\left\{\scat{A}_i\right\}$ be a finite family of finite categories with Euler characteristic. Then $\prod_{i} \scat{A}_i$ has Euler characteristic and $\EC{\prod_{i} \scat{A}_i} = \prod_{i} \EC{\scat{A}_i}$. 
	\end{thm}

	%AGAIN: DM or quiver?!
	This notion is easily understood in category-theoretic terms. In terms of quivers, we might think of this in terms of a tensor product of multigraphs as defined in \cite{WEIKRON}: the edges of the multigraph are still given by ordered pairs, but now the adjacency matrix of the multigraph is given by the Kronecker product of the adjacency matrices which make it up. 
	
	\par
	But why is the quantity we have been discussing known as the Euler characteristic? To justify such an expression, it ought to satisfy at least two properties:
	\begin{enumerate}
	\setlength{\itemindent}{.5in}
	\item[\textbf{P1}] Agreement with previous definitions of Euler characteristic
	\item[\textbf{P2}] Invariance under equivalence
	\end{enumerate}
	It turns out that this is, in fact, the case: both properties are satisfied. When we move to higher categories, our goal will be to satisfy these two properties again. Because our extension of Euler characteristic to higher categories is analogous to Leinster's proof that \textbf{P2} is satisfied, we now state his theorem about the Euler characteristic of a finite category being invariant under (categorical) equivalence. We give a proof which is a slight modification of Leinster's original. Several ingredients of Leinster's original proof have more prominent roles here, so we state them as lemmata. 
	
	\begin{lem}
	\label{thm: isomorphic objects have bijective hom sets}
	If $x$ and $y$ are isomorphic objects in a category $\scat{A}$ then $|\Hom\left(x,z\right)| = |\Hom\left(y,z\right)|$ for all $z \in \ob\left( \scat{A}\right)$. Similarly, $|\Hom\left(w,x\right)| = |\Hom\left(w,y\right)|$ for all $w \in \ob\left( \scat{A}\right)$. 
	\end{lem}
	\begin{proof}
	We will only prove the first claim, that $|\Hom\left(x,z\right)| = |\Hom\left(y,z\right)|$ for all $z \in \ob\left( \scat{A}\right)$. \\
	For isomorphic $x$ and $y$, we have $f\colon x \to y$ and $g \colon y \to x$ such that $f \circ g = 1_{y}$ and $g \circ f = 1_{x}$. \\
	We have a mapping $G\colon \Hom\left(x,z\right) \to \Hom\left(y,z\right)$ given by $G\left(\rho\right) = \rho \circ g$. \\
	Now, $G$ is surjective: for any $\pi \in \Hom\left(y,z\right)$, $\pi \circ f \in \Hom\left(x,z\right)$ will get mapped to $\pi$. \\
	Additionally, $G$ is injective: if $\rho \circ g = \pi \circ g$ then we have \[\rho = \rho \circ \left(g \circ f\right) = \left(\rho \circ g\right) \circ f = \left(\pi \circ g\right) \circ f = \pi \circ \left(g \circ f \right) = \pi.\] \\
	Since z was arbitrary, the claim is proven, as $G\colon \Hom\left(x,z\right) \to \Hom\left(y,z\right)$ is a bijection. The other claim follows similarly. 
	\end{proof}
	
	Another way of stating this lemma is that the function $|\Hom|\colon \ob\left( \scat{A}\right) \times \ob\left(\scat{A}\right) \to \mathbb{N}$ is constant on isomorphic classes of $\scat{A}$. This lemma will play an important role in our main theorem of invariance under equivalence and also helps facilitate our next lemma. 
	
	\begin{lem}
	\label{thm: constant weighting}
	If a category $\scat{A}$ has a weighting (resp. coweighting) $k^{\bullet}$  then we can find a weighting (resp. coweighting) $\alpha^{\bullet}$ that is constant on isomorphism classes.
	\end{lem}
	\begin{proof}
	Again, we will only focus on the case of weighting because the case of the coweighting is so similar. The basic idea is that we can partition $\ob\left( \scat{A}\right)$ into isomorphism classes and get a ``sum of weighting'' for each of these classes. If we divide each of these sums by the number of objects in the isomorphism class, we will have a single number associated to the each isomorphism class (and thus to each element in the isomorphism class). We can then leverage the fact that $|\Hom|\colon \ob\left( \scat{A}\right) \times \ob\left(\scat{A}\right) \to \mathbb{N}$ is constant on isomorphism classes to show that we will, in fact, get a weighting. \par
	In more detail: for each $x \in \ob\left( \scat{A}\right)$, let $\alpha^{x} = \frac{\sum_{y \cong x}k^{y}}{C_{x}}$, where $C_x$ is the number of objects in the isomorphism class of $x$. Then $\alpha^{\bullet}$ is constant on isomorphic objects. Suppose $\scat{A}$ has $m$ isomorphism classes. 
	Then for any $x \in \ob\left( \scat{A}\right)$, 
	\begin{align*}
	\sum_{y \in \ob\left( \scat{A}\right)}|\Hom\left(x,y\right)|\alpha^{y} &= \sum_{i=1}^{m}\sum_{y \cong y_{i}} |\Hom\left(x,y_{i}\right)|\alpha^{y_{i}} \\
	&= \sum_{i=1}^{m} C_{y_{i}}|\Hom\left(x,y_{i}\right)|\alpha^{y_{i}} &\\
	&= \sum_{i=1}^{m} C_{y_{i}}|\Hom\left(x,y_{i}\right)|\frac{1}{C_{y{i}}}\sum_{y \cong y_{i}} k^{y} \\
	&= \sum_{i=1}^{m} |\Hom\left(x,y_{i}\right)|\sum_{y \cong y_{i}}k^{y} \\
	&= \sum_{i=1}^{m} \sum_{y \cong y_{i}} |\Hom\left(x,y\right)|k^{y} \\
	&= \sum_{y \in \ob\left( \scat{A}\right)}|\Hom\left(x,y\right)|k^{y} = 1, \\
	\end{align*}
	the last equality coming because $k^{\bullet}$ is a weighting, and the second to last equality coming because $|\Hom\left(x,y_{i}\right)| = |\Hom\left(x,y\right)|$ for $y$ isomorphic to $y_{i}$ by Lemma \ref{thm: isomorphic objects have bijective hom sets}. This was all we had to do to show that $\alpha^{\bullet}$ is a weighting.
	\end{proof}
	
	The last proof gives some hint of what to expect from the proof that equivalent finite categories with Euler characteristic have the same Euler characteristic. In particular, we will be manipulating sums like in the above proof. When we deal with equivalent categories, we will need one additional ingredient which tells us that we can switch between categories (and their respective weightings) while manipulating the sums. This comes in the form of a somewhat elementary lemma. 
	
	\begin{lem}
	\label{thm: equivalence of categories preserves isomorphism classes}
	If \kern-3em $\xymatrix{&\scat{A} \ar@<.5ex>[r]^{F} &\scat{B} \ar@<.5ex>[l]^{G}}$ is an equivalence of categories then $F$ preserves isomorphism classes of $\scat{A}$ and $G$ preserves isomorphism classes of $\scat{B}$. 
	\end{lem}
	
	We are now ready to state and prove Leinster's theorem about the invariance of the Euler characteristic. Again, the proof is ours, though it is only a slight modification of that which appears in \cite{LEICAT}. 
	
	\begin{thm}[\cite{LEICAT} Prop 2.4]
	\label{thm: equiv cats have same ec}
	If \kern-3em $\xymatrix{&\scat{A} \ar@<.5ex>[r]^{F} &\scat{B} \ar@<.5ex>[l]^{G}}$ is an equivalence of finite categories and $\scat{B}$ has Euler characteristic, then $\scat{A}$ has Euler characteristic and $\EC{\scat{A}} = \EC{\scat{B}}$. 
	\end{thm}
	\begin{proof}
	Since $\scat{B}$ has Euler characteristic, it has a weighting $\ell^\bullet$ which is constant on isomorphic classes. We will want to show that we can come up with a weighting $k^{\bullet}$ on $\scat{A}$ that sums to $\sum_{b \in \ob\left( \scat{B}\right)}\ell^b$. With that in mind, define $k^a = \frac{1}{C_a}\sum_{b\colon b\cong F\left(a\right)}\ell^b$, where $C_{a}$ is, as before, the number of objects in the isomorphism class containing $a$. \\
	We have 
	\begin{align*}
	\sum_{a \in ob \ \scat{A}}k^a &= \sum_{a \in ob \scat{A}}\frac{1}{C_a}\sum_{b\colon b\cong F\left(a\right)}\ell^b \\
	&= \sum_{i=1}^m \sum_{a \cong a_i} \frac{1}{C_{a_i}}\sum_{b\colon b\cong F\left(a_i\right)}\ell^b \\ 
	&= \sum_{i=1}^m C_{a_i} \frac{1}{C_{a_i}}\sum_{b\colon b\cong F\left(a\right)}\ell^b \\ 
	&= \sum_{i=1}^m \sum_{b\colon b\cong F\left(a\right)}\ell^b \\
	&= \sum_{b \in ob \ \scat{B}}\ell^b \\
	\end{align*}
	as we were to show. We also need to show that a similar construction of coweighting works, but it is essentially identical to the case of the weighting which we just considered. 
	\end{proof}
	
	\subsection{Extending the Euler characteristic to weak n-categories}
	Having just run through the construction of Euler characteristic for a certain class of finite categories (namely, those which admit both a weighting and coweighting), we observe some analogies between categories and bicategories that could be of use in extending the Euler characteristic to certain finite bicategories. In particular, we have the following table which draws possible analogies which we hope can allow us to use theorems and proofs almost identical to those in the previous section.

	\begin{center}
	\begin{tabular}{|c | x{6cm}| x{6cm}|}
	\hline
	Analogy & 1-categories & bicategories\\ \hline
	\textbf{A1} & $\zeta\left(x,z\right)$ & $\EC{\scat{A}\left(x,z\right)}$\\ \hline
	\textbf{A2} & Isomorphic $x$ and $y$ have $|\Hom\left(x,z\right)| = |\Hom\left(y,z\right)|$ and $|\Hom\left(z,x\right)| = |\Hom\left(z,y\right)|$ for all objects $z$. & Internally equivalent $x$ and $y$ have $\EC{\scat{A}\left(x,z\right)} = \EC{\scat{A}\left(y,z\right)}$ and $\EC{\scat{A}\left(z,x\right)} = \EC{\scat{A}\left(z,y\right)}$for all objects $z$. \\ \hline
	\textbf{A3} & We can find weighting and coweighting constant on isomorphic classes. & We can find weighting and coweighting constant on classes of internally equivalent objects. \\ \hline
	\textbf{A4} & Functors preserve isomorphism classes. & Morphisms preserve internal equivalence classes. \\ \hline
	\end{tabular}
	\end{center}
	
	We now set out to justify each of these analogies. We take for granted that analogy $\textbf{A1}$ is suitable, for it will be our starting point. In particular, we will begin with the following definition:
	
	\begin{definition}
	The \emph{adjacency matrix} of a finite bicategory $\scat{A}$ under some total ordering $\left\{a_1, a_2, . . ., a_n\right\}$ is \[\emph{M}\left(\scat{A}\right) = 	\begin{bmatrix}
    	\EC{\scat{A}\left(a_1,a_1\right)}      & \dots & \EC{\scat{A}\left(a_1,a_n\right)} \\
   	 \vdots & \ddots & \vdots \\
   	 \EC{\scat{A}\left(a_n,a_1\right)}     & \dots & \EC{\scat{A}\left(a_n,a_n\right)}
	\end{bmatrix}. \]
	\end{definition}
	
	Note that this adjacency matrix is only defined when all of the Euler characteristics exist. Thus the following definition is already somewhat restricted: 
	
	\begin{definition}
	\label{defn: bicategory ec}
		If a bicategory $\scat{A}$ admits both a weighting $k^{\bullet}$ and coweighting $k_{\bullet}$ then its \emph{Euler characteristic} is $\EC{\scat{A}} =  \sum_{i=1}^{n} k^{a_{i}} = \sum_{i=1}^{n} k_{a_{i}}$. 
	\end{definition}
	
	The following example shows that our definition of Euler characteristic for bicategories agrees with Leinster's definition for finite categories.
	
	\begin{exmp}
		\label{calc: 1-cat as 2-cat}
			If we start with a finite category $\scat{A}$ with Euler characteristic, we can view it as a bicategory $\widetilde{\scat{A}}$ with $\scat{A}\left(x,y\right)$ a discrete category for each pair of objects $x,y$ in $\scat{A}$. In this case, the adjacency matrix of $\scat{A}$ is identical in each entry to the adjacency matrix of the $\widetilde{\scat{A}}$. We therefore have our definition of Euler characteristic agreeing with Leinster's definition for finite categories. 
	\end{exmp}
	
	In addition to this agreement with previous definition, using the adjacency matrix approach above should extend some of the nice properties the Euler characteristic had for finite categories. In particular, since the Euler characteristic for finite categories is invariant under equivalence and works well with products and coproducts, the adjacency matrix adapts quite nicely for products and coproducts of bicategories.\footnote{Going one level lower, we might think of the Euler characteristic for finite categories working well with products and coproducts in $\mathbf{Cat}$ specifically because cardinality works well with products and coproducts in $\mathbf{Set}$ and the adjacency matrix of a finite category is given in terms of cardinality. The analogy is even more complete: recall from Example \ref{exmp: set as category} that Euler characteristic for a finite category agrees with the notion of cardinality of a finite set.}
	
	We should note that our analogy \textbf{A2} really should reflect the restriction that all Euler characteristics in the adjacency matrix must exist. Keeping in mind this necessary restriction, we can confirm the validity of \textbf{A2} as a corollary to the following lemma:
	
	\begin{lem}
	If $x$ and $y$ are internally equivalent objects of a bicategory $\scat{A}$ then $\scat{A}\left(x,z\right)$ and $\scat{A}\left(y,z\right)$ are equivalent categories for all $z \in \ob\left( \scat{A}\right) $ and, similarly, $\scat{A}\left(w,x\right)$ and $\scat{A}\left(w,y\right)$ are equivalent categories for all $w \in \ob\left( \scat{A}\right)$. 
	\end{lem}
	\begin{proof}
	The two equivalences are essentially identical, so we only write out full details for the first. 
	In the case of the first equivalence, we need to come up with functors $F\colon \scat{A}\left(x,z\right)\to\scat{A}\left(y,z\right)$ and $G\colon \scat{A}\left(y,z\right)\to\scat{A}\left(x,z\right)$ which make up an equivalence of categories, i.e. $F \circ G \cong Id_{\scat{A}\left(x,z\right)}$ and $G \circ F \cong Id_{\scat{A}\left(y,z\right)}$. \par
	Since $x$ and $y$ are internally equivalent, we have 1-cells $f\colon x \to y$ and $g\colon y \to x$ such that $g \circ f \cong 1_{x}$ and $f \circ g \cong 1_{y}$ by 2-cells, say $\alpha$ and $\beta$ respectively. \par
	Let $F$ be given by 
	\begin{itemize}
	\item $\rho \mapsto \rho \circ g$
	\item $\delta \mapsto \delta \ \mycirc{h} \ \widetilde{\textbf{1}}_{g}$
	\end{itemize}
	Then $F\left(\widetilde{\textbf{1}}_{\rho}\right) = \widetilde{\textbf{1}}_{\rho} \mycirc{h} \widetilde{\textbf{1}}_{g} = \widetilde{\textbf{1}}_{\rho \circ g} = \widetilde{\textbf{1}}_{F\left(\rho\right)}$ and
	\[ F\left(\delta_{2} \ \mycirc{v}_{xz} \ \delta_{1}\right) = \left(\delta_{2} \ \mycirc{v} \ \delta_{1}\right)\mycirc{h} \ \widetilde{\textbf{1}_{g}} = \left(\delta_{2} \mycirc{h} \delta_{1}\right) \  \mycirc{v}_{yz}  \ \left(\delta_{1} \mycirc{h} \widetilde{\textbf{1}_{g}}\right) = F\left(\delta_{2}\right) \ \mycirc{v}_{yz} \ F\left(\delta_{1}\right). \]
	That is, $F$ is indeed a functor. We may similarly define a functor $G$ by
	\begin{itemize}
	\item $\pi \mapsto \pi \circ f$
	\item $\epsilon \mapsto \epsilon \ \mycirc{h} \ \widetilde{\textbf{1}}_{f}$
	\end{itemize}
	Furthermore, we now have the following the following commutative diagrams: \\
	$\xymatrix@=5em{& \rho \ar[d]^{\delta} \\ & \rho^{\prime}}$
	$\xymatrix@=5em{& \left(G \circ F\right)\left(\rho\right)\ar[d]^{\left(G\circ F\right)\left(\delta\right)} \ar[r]^{\eta_{\rho}} & Id_{\scat{A}\left(x,z\right)}\left(\rho\right) \ar[d]^{ Id_{\scat{A}\left(x,z\right)}\left(\delta\right)} & \\ & \left(G\circ F\right)\left(\rho^{\prime}\right) \ar[r]^{\eta_{\rho^{\prime}}} &  Id_{\scat{A}\left(x,z\right)}\left(\rho^{\prime}\right)}$\\
	$\xymatrix@=5em{& \pi \ar[d]^{\epsilon} \\ & \pi^{\prime}}$
	$\xymatrix@=5em{&\left(F \circ G\right)\left(\pi\right)\ar[d]^{\left(F\circ G\right)\left(\epsilon\right)} \ar[r]^{\nu_{\pi}} & Id_{\scat{A}\left(y,z\right)}\left(\pi\right) \ar[d]^{ Id_{\scat{A}\left(y,z\right)}\left(\epsilon\right)} & \\ & \left(F \circ G\right)\left(\pi^{\prime}\right) \ar[r]^{\nu_{\pi^{\prime}}} &  Id_{\scat{A}\left(y,z\right)}\left(\pi^{\prime}\right)}$\\
	given by $\eta\left(\rho\right) = r_{\rho} \ \mycirc{v}_{xz} \ \left(\widetilde{\textbf{1}}_{\rho} \ \mycirc{h} \ \alpha \right) \ \mycirc{v}_{xz} \ A_{\rho gf}$ and $\nu_{\pi}=  r_{\pi} \ \mycirc{v}_{yz} \ \left(\widetilde{\textbf{1}}_{\pi} \ \mycirc{h} \ \beta \right) \ \mycirc{v}_{yz} \ A_{\pi fg}$. Recall the role of the $r$ and $A$ as right unitors and associators, which are (2-)isomorphisms. \\
	We therefore have the necessary natural isomorphisms to conclude that $\scat{A}\left(x,z\right) \cong \scat{A}\left(y,z\right)$ for arbitrary $z \in \ob\left( \scat{A}\right)$. The proof that $\scat{A}\left(w,x\right) \cong \scat{A}\left(w,y\right)$ for arbitrary $w \in \ob\left( \scat{A}\right)$ is similar.  
	\end{proof}
	
	\begin{cor}
	\label{thm: equiv objects have same ec}
	Let $x$ and $y$ be internally equivalent objects of a bicategory, and let $z$ be any object in the same bicategory. Then $\mathbf{\left(1\right)}$ $\EC{\scat{A}\left(x,z\right)} = \EC{\scat{A}\left(y,z\right)}$ whenever one of these quantities exists and $\mathbf{\left(2\right)}$ $\EC{\scat{A}\left(z,x\right)} = \EC{\scat{A}\left(z,y\right)}$ whenever one of these quantities exists. 
	\end{cor}
	
	Note that the above corollary also depends on Theorem \ref{thm: equiv cats have same ec}. Our analogies \textbf{A3} and \textbf{A4} are justified by theorems whose proofs are essentially identical to those of their 1-category analogues, Lemma \ref{thm: constant weighting} and Lemma \ref{thm: equivalence of categories preserves isomorphism classes}, respectively.
	
	\begin{lem}
	\label{thm: weighting on homotopy classes}
	If a bicategory $\scat{A}$ has a weighting $k^{\bullet}$ (resp. coweighting $k_{\bullet}$) then it has a weighting $\alpha$ (resp. coweighting $\beta$) which are constant on internal equivalence classes.
	\end{lem}
	
	\begin{lem}
	\label{thm: equivalence of bicategories preserves homotopy classes}
	If \kern-3em $\xymatrix{&\scat{A} \ar@<.5ex>[r]^{F} &\scat{B} \ar@<.5ex>[l]^{G}}$ is an equivalence of bicategories then $F$ preserves internal equivalence classes of $\scat{A}$ and $G$ preserves internal equivalence classes of $\scat{B}$. 
	\end{lem}
	\begin{proof}
	For brevity, we omit actual details. The result follows from definition of equivalence of bicategories. The proof is somewhat similar to a standard proof that equivalence of categories preserves isomorphism classes, but it is messier. 
	\end{proof}
	
	Looking back at our table, we see that all of our analogies are justified. We therefore have the following theorem which tells us that Euler characteristic, as defined in Definition \ref{defn: bicategory ec}, is invariant under equivalence of bicategories.
	
	\begin{thm}
	Let $\scat{A}$ and $\scat{B}$ be equivalent bicategories. Then either both have Euler characteristic, in which case $\EC{\scat{A}} = \EC{\scat{B}}$, or neither has Euler characteristic.
	\end{thm}
	
	The proof of this theorem is almost identical to that of Theorem \ref{thm: equiv cats have same ec}, modifying the proof to fit the bicategories context as appropriate. 
	
	By generalizing the work above, a sketch falls out for the case of weak n-categories for any positive integer n. If we can view a weak n-category as having a weak (n-1)-category between any two objects then we can use \textbf{A1}. If we have some appropriate notion of equivalence which ultimately resides in our (n-1)-category then we have hopes of using \textbf{A2-A4}. However, the definition of weak n-category is not agreed upon, and even using an approximation of Leinster's definition of Bicategory to move to $n \geq 3$ yields a large collection of identities, diagrams, and equivalences which are simply too numerous to reasonably include here. We thus conclude our discussion of weak n-categories with the above sketch, and we shift our attention to the case of $\left(\infty,1\right)-$categories.
	
	\section{$\left(\infty,1\right)$-categories}
	\subsection{Introduction to $\left(\infty,1\right)$-categories}
	When we move to $\infty$-categories, we have a choice to make, for there are a number of ``competing'' definitions for what an $\infty$-category is. We will attempt to develop the Euler characteristic for the $\left(\infty,1\right)$-categories defined by Lurie in \cite{LURIEHTT}. Before we unveil our notion of Euler characteristic, however, we need to provide some background information on what these $\infty$-categories are. To put it briefly, they are special simplicial sets, so we shall now provide some basic information about simplicial sets and also introduce the appropriate notion of ``specialness.'' To begin, we present definitions of abstract simplicial complexes and ordered simplicial complex. A more complete reference is \cite{DWYNOTES}. 
	
	\begin{definition}
	\label{defn: simplicial complex}
	An \emph{abstract simplicial complex} is a pair $K = \left(V_{K},S_{K}\right)$ comprised of a set $V_{K}$ of vertices and a collection $S_{K}$ of non-empty, finite subsets (i.e. simplices) of $V_{K}$ such that the following two properties hold:
	\begin{itemize}
	\item If $\varnothing \neq \sigma^{\prime} \subset \sigma$ and $\sigma \in S_{K}$ then $\sigma^{\prime} \in S_{K}$. We call $\sigma^{\prime}$ a \emph{face} of $\sigma$.
	\item $\left\{v\right\} \in S_{K}$ for every $v \in V_{K}$. 
	\end{itemize}
	An \emph{ordered simplicial complex} is an abstract simplicial complex $K = \left(V_{K},S_{K}\right)$ in which $V_{K}$ has a partial ordering which is a total ordering on each simplex $\sigma \in S_{K}$. 
	\end{definition}
	
	\begin{exmp}
	\label{exmp: standard abstract n-simplex}
	For any $n \in \mathbb{Z}_{\geq 0}$, let $\underline{n} = \left\{0, 1, ..., n\right\}$ equipped with its usual ordering. Then $[n] = \left(\underline{n},2^{\underline{n}} \setminus \left\{\varnothing\right\}\right)$ is an ordered simplicial complex called the \emph{standard abstract n-simplex}.
	\end{exmp}
	
	The standard abstract n-simplex may be contrasted with the \emph{standard topological n-simplex}.

	\begin{definition}
	\label{defn: standard n-simplex}
	The \emph{standard n-simplex} $\Delta_{n}$ is the topological space given by the set \[\left\{\left(x_{0},x_{1},...,x_{n}\right) \in \mathbb{R}^{n+1} |  \sum_{i=0}^{n} x_{i} = \allowbreak 1, x_{i} \geq 0\right\} \] under $\mathbb{R}^{n+1}$'s Euclidean subspace topology. 
	\end{definition}
	
	The first few standard n-simplices are actually quite familiar. The standard 0-simplex $|\Delta_{0}|$ is just a single point, $\left\{1\right\}$; the standard 1-simplex is just the closed line segment connecting $\left(0,1\right)$ and $\left(1,0\right)$; the standard 2-simplex is the triangle with vertices $\left(0,0,1\right), \left(0,1,0\right),$ and $\left(1,0,0\right)$; the standard 3-simplex is a tetrahedron with vertices $\left(0,0,0,1\right), \left(0,0,1,0\right), \left(0,1,0,0\right),$ and  $\left(1,0,0,0\right)$. We see that these are relatively simple spaces to construct. However, we would like to relate them to other topological spaces, and we will actually find that a generalization of simplicial complexes is a good step to doing so. 
	
	\begin{definition}
	\label{defn: n-th singular complex of X}
	Let $X$ be a topological space. We denote by \emph{$\mathrm{Sing}_{n}\left(X\right)$} the set $\Hom_{\mathbf{Top}}\left(\Delta_{n},X\right)$, i.e. the continuous maps from $\Delta_{n}$ to $X$. 
	\end{definition}
	
	It turns out that maps between our standard abstract n-simplices actually induce maps between our $\mathrm{Sing}_{n}\left(X\right)$. In outlining the details for this process, we prefer to forget $[n]$'s status as a simplicial complex and just focus on its vertex set, $\underline{n}$. We can think of $\left\{\underline{n}\right\}_{n \geq 0}$ as the object set of a full subcategory $\Delta$ of $\mathbf{Set}$ if we take the morphisms between two of these objects to be the order-preserving maps. We will focus on order-preserving injections $\delta^{i}\colon\underline{m}\to\underline{m+1}$ which omit a fixed $i$ and order-preserving surjections $\sigma^{i}\colon\underline{m+1}\to\underline{m}$ which repeat a fixed $i$. 
In detail, $\delta^{i}$ is given by \[ k \mapsto \left\{ \begin{array}{ll} k & k < i \\ k + 1 & k \geq i, \end{array}\right. \] and $\sigma^{i}$ is given by \[ k \mapsto \left\{ \begin{array}{ll} k & k < i \\ k - 1 & k \geq i. \end{array}\right. \] 
	
	Every morphism in $\Delta$ can be generated by composing degeneracy and face maps. Additionally, these maps induce ways of taking ``faces'' and ``degeneracies'' in the geometric context. For instance, for some fixed $i$ and a topological space $X$, $\delta^{i}$ induces a \emph{face map} $\delta_{i}\colon \mathrm{Sing}_{m+1}\left(X\right) \to \mathrm{Sing}_{m}\left(X\right)$ given by $\delta_{i} = f \circ |\delta^{i}|$ for each $f \in \mathrm{Sing}_{m}\left(X\right)$. Similarly, $\sigma^{i}$ induces a \emph{degeneracy map} $\sigma_{i}\colon \mathrm{Sing}_{m}\left(X\right) \to \mathrm{Sing}_{m+1}\left(X\right)$ given by $\sigma_{i} = f \circ |\sigma^{i}|$ for each $f \in \mathrm{Sing}_{m+1}\left(X\right)$. 
	 In both of these cases, we have ``geometrically realized'' our $\delta^{i}$ or $\sigma^{i}$ in order to have the appropriate compositions well-defined.
	
	The astute reader may have noticed another functor hanging around in this discussion: $\mathrm{Sing}\left(X\right)\colon \Delta \to \mathbf{Set}$, which is given by 
	\begin{itemize}
	\item $\underline{n} \mapsto \Hom_{\mathbf{Top}}\left(\Delta_{n},X\right)$
	\item $h \in \Hom_{\Delta}\left(\underline{m},\underline{m'}\right) \mapsto |h| \in \Hom_{\mathbf{Top}}\left(\Delta_{m},\Delta_{m'}\right)$. 
	\end{itemize}
	
	We can also think of any standard abstract n-simplex $[n]$ as a (contravariant) functor $[n]\colon \Delta \to \mathbf{Set}$ by the so-called ``Hom-functor'':
	\begin{itemize}
	\item $\underline{m} \mapsto \Hom_{\Delta}\left(\underline{m},\underline{n}\right)$
	\item $\left(f: \underline{m} \to \underline{m'}\right) \mapsto \left(\Hom\left(f,\underline{n}\right)\colon \Hom\left(\underline{m'},\underline{n}\right) \to \Hom\left(\underline{m},\underline{n}\right)\right)$, which maps $\left(g\colon \underline{m'} \to \underline{n}\right)$ to  $g \circ f$.
	\end{itemize}
	
	Both of these functors, $[n]$ and $\mathrm{Sing}\left(X\right)$, are examples of what are called simplicial sets.
	
	\begin{definition}
	A \emph{simplicial set} X is a functor $X\colon\Delta^{op} \to \mathbf{Set}$
	\end{definition}
	
	Because of the structure of $\Delta$, simplicial sets generally have face maps and degeneracy maps induced by the $\sigma^{i}$ and $\delta^{i}$ in $\Delta$. By specifying an element to omit, $\delta^{i}$ determines a way of moving down a ``dimension'' via $\delta_{i}$. Similarly, by repeating an element, $\sigma^{i}$ determines a way of moving up a ``dimension'' in a degenerate way. These face maps and degeneracy maps satisfy the following ``simplicial identities'':
	
	\begin{enumerate}
	\item $\delta_{i}\delta_{j} = \delta_{j-1}\delta_{i}$, \kern+3em $i < j$
	\item $\delta_{i}\sigma_{j} = \sigma_{j-1}\delta_{i}$, \kern+3em $ i < j$
	\item $\delta_{j}\sigma_{j} = \delta_{j+1}\sigma_{j} = 1$
	\item $\delta_{i}\sigma_{j} = \sigma_{j}\delta_{i-1}$, \kern+3em $ i > j + 1$
	\item $\sigma_{i}\sigma_{j} = \sigma_{j+1}\sigma_{i}$, \kern+3em $i \leq j$
	\end{enumerate}
	
	As pointed out in \cite{GOESHT}, from whence the above identities were reproduced, specifying a simplicial set $X$ amounts to saying what the object function of $X$ does and what the face and degeneracy maps (i.e. the maps that satisfy the five identities above) are. Simplicial sets are objects of the functor category $\mathbf{sSet} = [\Delta^{op},\mathbf{Set}]$, whose objects are functors from $\Delta^{op}$ to $\mathbf{Set}$ and whose morphisms are natural transformations between these functors. The objects of this category, i.e. the simplicial sets, are sometimes referred to as \emph{presheaves} on $\Delta$ and the morphisms are sometimes called \emph{simplicial maps}. 
	
	\begin{exmp}
	\label{exmp: one point simplicial set}
	The one-point space $\left\{*\right\}$ can be thought of as the functor which sends every $\underline{m}$ to a one-element set.
	\end{exmp}
	
	\begin{exmp}
	\label{exmp: nerve of a category}
	For any category $\scat{A}$, we can associate a simplicial set to it, called the \emph{Nerve of} $\scat{A}$ and denoted $\mathcal{N}\left(\scat{A}\right)$. This simplicial set can be specified by the following  information.
	\begin{itemize}
	\item $\mathcal{N}\left(\scat{A}\right)_{n} = \left\{\text{$n$-paths in } \scat{A}\right\}$, i.e. the collection of all sequences
	\[ \xymatrix{x_{0} \ar[r]^{f_{1}} & x_{1} \ar[r]^{f_{2}} & x_{2} \ar[r] & . . . \ar[r]^{f_{n-1}} & x_{n-1} \ar[r]^{f_{n}} & x_{n}}\]
	 of composable morphisms. 
	\item The degeneracy map $\sigma_{i}\colon \mathcal{N}\left(\scat{A}\right)_{n} \to \mathcal{N}\left(\scat{A}\right)_{n+1}$ corresponds to adding an identity morphism between $x_{i}$ and $x_{i+1}$.
	\item The face map $\delta{i}\colon\mathcal{N}\left(\scat{A}\right)_{n} \to \mathcal{N}\left(\scat{A}\right)_{n-1}$ corresponds to composing $f_{i}$ and $f_{i-1}$ if $i$ is not $0$ or $n$. If $i$ is $0$ or $n$ then it corresponds to simply removing the morphism $f_{0}$ or $f_{n}$ from the diagram as appropriate. 
	\end{itemize}
	\end{exmp}
	
	The nerve of a category gives us a model for a space associated with $\scat{A}$, and some $\left(\infty,1\right)$-categories are intimately related to nerves, as we shall see from Theorem \ref{thm: nerve of a category characterization}. In order to state this theorem and define $\left(\infty,1\right)$-categories, we need to discuss one more important example of simplicial set: horns. 
	
	\begin{exmp}
	For fixed $n$, the \emph{k-horn} $\Lambda^{n}_{k}$ is a subsimplicial set of $[n]$ which is generated by $\left\{\delta_{i}\left(1_{n}\right) \colon i \neq k\right\}$, where $1_{n}\colon\underline{n}\to\underline{n}$ is the identity on $\underline{n}$. 
	\end{exmp}
	
	Since $\delta_{i}\colon X_{n}\to X_{n-1}$ can be thought of as taking $X_{n}$ to its $i$th face, the horn is essentially $[n]$ with all but its $i$th face. The horn is, in fact, a simplicial set, and it can be embedded into $[n]$ by an inclusion mapping $\iota$. \par
	For the purposes of $\infty$-categories, $[n]$ and its k-horns are the most important simplicial sets. Indeed, it is the existence of certain simplicial maps (and the corresponding commutative diagrams) which determine if a simplicial set is an $\left(\infty,1\right)$-category. To wit, 
	\begin{definition}
	\label{defn: infinity category}
	An \emph{$\left(\infty,1\right)$-category} $\scat{A}$ is a simplicial set $\scat{A}\colon \Delta^{op} \to \mathbf{Set}$ such that for each diagram
	\begin{center}
	$\xymatrix@=4em{ \Lambda^{n}_{k} \ar[r]^{\omega} \ar[d]_{\iota} & \scat{A} & \left(0 < k < n\right) \\ [n]  }$
	\end{center}
	there exists at least one simplicial map $\alpha\colon[n]\to\scat{A}$ such that 
	\begin{center}
	$\xymatrix@=4em{ \Lambda^{n}_{k} \ar[r]^{\omega} \ar[d]_{\iota} & \scat{A} & \left(0 < k < n\right) \\ [n] \ar[ur]_{\alpha}  }$
	\end{center}
	is a commutative diagram. 
	\end{definition}
	
	The $\left(\infty,1\right)$-categories are the objects of subcategory $\mathbf{\infty-Cat}$ of $\mathbf{sSet}$. Some $\left(\infty,1\right)$-categories are intimately connected to categories according to the following theorem, which tells us that the notion of $\left(\infty,1\right)$-category is a generalization of the notion of the nerve of a category.
	
	\begin{thm}
	\label{thm: nerve of a category characterization}
	An \emph{$\left(\infty,1\right)$-category} $\scat{A}$ is the nerve of a category $\mathcal{C}$ if and only if
	there exists a \emph{unique} simplicial map $\alpha\colon[n]\to\scat{A}$ such that 
	\begin{center}
	$\xymatrix@=4em{ \Lambda^{n}_{k} \ar[r]^{\omega} \ar[d]_{\iota} & \scat{A} & \left(0 < k < n\right) \\ [n] \ar[ur]_{\alpha}  }$
	\end{center}
	is a commutative diagram for all such diagrams.
	\end{thm}
	
	%The only thing that changes between the definition of the $\left(\infty,1\right)$-category and the theorem 
	
	\subsection{The Euler characteristic for $\left(\infty,1\right)$-categories}
	Our goal in this section is to find some function $\chi\colon\ob\left(\mathbf{\infty-Cat}\right)\to\mathbb{R} \cup \left\{+\infty\right\}$ that satisfies the following properties:
	\begin{enumerate}
	\item If $\scat{X}$ is the $\left(\infty,1\right)$-category which has all objects the empty set, $\EC{\scat{X}} = 0$.
	\item If $\scat{X}$ is the one-point space as an $\left(\infty,1\right)$-category, $\EC{\scat{X}} = 1$
	\item If $\scat{X}$ and $\scat{Y}$ are $\left(\infty,1\right)$-categories, $\EC{\scat{X} \prod \scat{Y}} = \EC{\scat{X}} * \EC{\scat{Y}}$.
	\item If $\scat{X}$ and $\scat{Y}$ are $\left(\infty,1\right)$-categories, $\EC{\scat{X} \coprod \scat{Y}} = \EC{\scat{X}} +  \EC{\scat{Y}}$.
	\item If $\scat{X}$ and $\scat{Y}$ are weakly equivalent $\left(\infty,1\right)$-categories and $\EC{\scat{X}}$ then $\EC{\scat{X}} = \EC{\scat{Y}}$.
	\item If $\scat{X}$ is the nerve of a category $\mathcal{C}$ then $\EC{\scat{X}} = \EC{\mathcal{C}}$. 
	\end{enumerate}
	
	It is likely that this list of desiderata is incomplete, and we ought to find out how the Euler characteristic behaves for other limits and colimits in order to have additional criteria. As it stands, it is unclear if such a function exists or if such a function is unique. If the function is not unique, what does this mean for Euler characteristic? Is there a \emph{right} choice for $\chi$? \par
	We do not have answers to these questions due to several difficulties. One difficulty is that, while Lurie presents a truncation process by which we can get \emph{Lurie n-categories} from $\left(\infty,1\right)$-categories, these n-categories do not correspond perfectly to the $n$-categories we worked with in the course of our study. Furthermore, we do not have a notion of Euler characteristic for Lurie n-categories, and we are also unaware of a process of switching between Lurie n-categories and those which our sketch addresses. 
	
	\section{Conclusions and Future Work}
	In these notes, we reviewed the construction of the Euler characteristic for finite categories and presented a generalized construction for finite weak n-categories. Since composition was only used to define the equivalence of objects and the equivalence of n-categories (which did not require strict equality with respect to units and associativity), we did not have to restrict ourselves to working with strict n-categories. As a tradeoff, however, we lost the straightforward definition of n-categories as categories enriched over the category of $(n-1)$-categories, which made it hard to work proofs out in full generality. We also briefly discussed $\left(\infty,1\right)$-categories and presented some of the obstacles to a construction of Euler characteristic. Ultimately, we were unable to present the Euler characteristic for $\left(\infty,1\right)$-categories. \par
	A number of directions for future research present themselves. Fixing a definition of (weak) n-category, it would be interesting for our sketch of the Euler characteristic for n-categories to be worked out rigorously in the general case. Additionally, the Euler characteristic for $\left(\infty,1\right)$-categories remains to be defined; more or less ambitious would be taking some of the na\"{\i}ve approaches and seeing when and how they agree or disagree with what the Euler characteristic of an $\left(\infty,1\right)$-category ``should be.''
	 
	\bibliography{references}
	\bibliographystyle{alpha}

\end{document}